\documentclass[12pt, a4paper]{article}

\usepackage{amssymb}
\usepackage{amsmath}
\usepackage{amsthm}
 \usepackage{dsfont}
 \usepackage{eufrak}

  \newtheorem{theorem}{Theorem}
 \newtheorem{definition}{Definition}
  \newtheorem{corollary}{Corollary}
    \newtheorem{lemma}{Lemma}    
    \newtheorem{proposition}{Proposition}

\DeclareMathOperator{\Div}{div}

\begin{document}

 \title{The Navier--Stokes equations in mixed-norm time-space parabolic Morrey spaces.}
\author{Pierre Gilles Lemari\'e-Rieusset\footnote{LaMME, Univ Evry, CNRS, Universit\'e Paris-Saclay, 91025, Evry, France; e-mail : pierregilles.lemarierieusset@univ-evry.fr}}
\date{}\maketitle

\begin{abstract} We discuss the Navier-Stokes equations with forces  in the mixed-norm time-space parabolic Morrey spaces of Krylov.
\end{abstract}


 
 \section{Introduction}
  In this paper, we consider global mild solutions of the Cauchy problem for the incompressible Navier--Stokes equations on the whole space $\mathbb{R}^3$.  More precisely, the Navier--Stokes equations we study are  
 \begin{equation}\label{NSE}\left\{
 \begin{split} &\partial_t \vec u =\Delta\vec u-\vec \nabla p-\vec u\cdot \vec \nabla\vec u + \vec f+\Div\mathbb{F}
\\& \Div\vec u=0
\\& \vec u(0,.)=  0
\end{split}\right.\end{equation}
where $\vec f$  and $\mathbb{F}$ are small enough in some critical  (homogeneous) mixed-norm time-space parabolic Morrey spaces of Krylov.

(Following Krylov \cite{KRY}), let us recall the definition of mixed-norm time-space parabolic Morrey spaces on $\mathbb{R}\times\mathbb{R}^d$:

\begin{definition}
  Let $1<p,q<+\infty$ and $\beta\in (0, \frac 2 p+\frac d q)$. The parabolic  Morrey spaces $E_{p,q,\beta}$ and $F_{p,q,\beta}$ are the spaces of locally integrable functions $f(t,x)$ on $\mathbb{R}\times\mathbb{R}^d$ such that $\|f\|_{E_{p,q,\beta}}<+\infty$ or $\|f\|_{F_{p,q,\beta}}<+\infty$, where
  $$ \|f\|_{E_{p,q,\beta}}=\sup_{\rho>0, t\in\mathbb{R}, x\in\mathbb{R}^d} \rho^{\beta-\frac 2 p-\frac d q} \left(\int_{\vert t-s\vert<\rho^2} (\int_{\vert x-y\vert<\rho} \vert f(s,y)\vert^q\, dy)^{\frac p q}\, ds\right)^{\frac 1 p}$$
  and   $$ \|f\|_{F_{p,q,\beta}}=\sup_{\rho>0, t\in\mathbb{R}, x\in\mathbb{R}^d} \rho^{\beta-\frac 2 p-\frac d q} \left(\int_{\vert x-y\vert<\rho} (\int_{\vert t-s\vert<\rho^2} \vert f(s,y)\vert^p\, ds)^{\frac  q p}\, dy\right)^{\frac 1 q}.$$
  \end{definition}
  Remark: when $p=q$, writing $r=\frac {d+2}\beta$ (so that $p<r<+\infty$), we see that $E_{p,p,\beta}=F_{p,p,\beta}=\mathcal{M}^{p,r}_2(\mathbb{R}\times\mathbb{R}^d)$, where $\mathcal{M}^{p,r}_2(\mathbb{R}\times\mathbb{R}^d)$ is the parabolic Morrey space studied in \cite{PGL2} in the context of Navier--Stokes equations.
  
  For a function defined on $(0,+\infty)\times\mathbb{R}^d$, we say that $f\in E_{p,q,\beta}$ or $f\in F_{p,q,\beta}$ if the function $f^\#$ defined by $ f^\#=f$ for $t>0$ and $f^\#=0$ for $t<0$ is such that  $ f^\#\in E_{p,q,\beta}$ or $ f^\#\in F_{p,q,\beta}$.
  
  Our main theorem is the following one:
  
  \begin{theorem}\label{theo1}
  Let $p,q\in (3,\infty)$ with $\frac 2 p+\frac 3 q>1$. \\ a) There exists $\epsilon_0=\epsilon_0(p,q)>0$ and $C_0=C_0(p,q)>0$ such that, if $\vec f\in E_{p/3,q/3,3}$ with $\Div\vec f=0$  and $\mathbb{F}\in E_{p/2,q/2,2}$ and if
  $$ \|\vec f\|_{E_{p/3,q/3,3}}+\| \mathbb{F}\|_{E_{p/2,q/2,2}} <\epsilon_0,$$ then the Navier--Stokes equations  (\ref{NSE}) have a global mild solution $\vec u\in E_{p,q,1}$ and $\vec\nabla\otimes\vec u\in E_{p/2,q/2,2}$ with 
$$ \|\vec u\|_{E_{p,q,1}}+\|\vec\nabla\otimes\vec u\|_{ E_{p/2,q/2,2}}\leq C_0 (\|\vec f\|_{E_{p/3,q/3,3}}+\| \mathbb{F}\|_{E_{p/2,q/2,2}}).$$
b)  There exists $\epsilon_0=\epsilon_0(p,q)>0$ and $C_0=C_0(p,q)>0$ such that, if $\vec f\in F_{p/3,q/3,3}$ with $\Div\vec f=0$  and $\mathbb{F}\in F_{p/2,q/2,2}$ and if
  $$ \|\vec f\|_{F_{p/3,q/3,3}}+\| \mathbb{F}\|_{F_{p/2,q/2,2}} <\epsilon_0,$$ then the Navier--Stokes equations   
(\ref{NSE}) have a global mild solution $\vec u\in F_{p,q,1}$ and $\vec\nabla\otimes\vec u\in F_{p/2,q/2,2}$ with 
$$ \|\vec u\|_{F_{p,q,1}}+\|\vec\nabla\otimes\vec u\|_{ F_{p/2,q/2,2}}\leq C_0 (\|\vec f\|_{F_{p/3,q/3,3}}+\| \mathbb{F}\|_{F_{p/2,q/2,2}}).$$
  \end{theorem}
  
  When the force $\vec f$ is equal to $0$, we may lower the values of $p$ and $q$:
  \begin{theorem}\label{theo2}
  Let $p,q\in (2,\infty)$ with $\frac 2 p+\frac 3 q>1$. \\ a) There exists $\epsilon_0=\epsilon_0(p,q)>0$ and $C_0=C_0(p,q)>0$ such that, if   $\mathbb{F}\in E_{p/2,q/2,2}$ and 
  $$ \| \mathbb{F}\|_{E_{p/2,q/2,2}} <\epsilon_0,$$ then the Navier--Stokes equations  (\ref{NSE}) (with $\vec f=0$) have a global mild solution $\vec u\in E_{p,q,1}$ and $\vec\nabla\otimes\vec u\in E_{p/2,q/2,2}$ with 
$$ \|\vec u\|_{E_{p,q,1}}+\|\vec\nabla\otimes\vec u\|_{ E_{p/2,q/2,2}}\leq C_0 \| \mathbb{F}\|_{E_{p/2,q/2,2}}.$$
b)  There exists $\epsilon_0=\epsilon_0(p,q)>0$ and $C_0=C_0(p,q)>0$ such that, if   $\mathbb{F}\in F_{p/2,q/2,2}$ and if
  $$  \| \mathbb{F}\|_{F_{p/2,q/2,2}} <\epsilon_0,$$ then the Navier--Stokes equations   
(\ref{NSE})  (with $\vec f=0$) have a global mild solution $\vec u\in F_{p,q,1}$ and $\vec\nabla\otimes\vec u\in F_{p/2,q/2,2}$ with 
$$ \|\vec u\|_{F_{p,q,1}}+\|\vec\nabla\otimes\vec u\|_{ F_{p/2,q/2,2}}\leq C_0 \| \mathbb{F}\|_{F_{p/2,q/2,2}}.$$
  \end{theorem}
  \section{Heat equation} 
  The proof of   theorems \ref{theo1} and \ref{theo2} will rely on the theory of the heat equation $\partial_t u-\Delta  u=f$ on $\mathbb{R}\times\mathbb{R}^d$ with the boundary condition $u=0$ at infinity. This condition will be defined as follows:
  
  \begin{definition} Let $u$ be a tempered distribution on  $\mathbb{R}\times\mathbb{R}^d$. Then $u=0$ at infinity if for every $\theta>0$ one has $e^{\theta^2(\partial_t^2-\Delta^2)}u\in L^\infty(\mathbb{R}\times\mathbb{R}^d)$ and $\lim_{\theta\rightarrow +\infty} \|e^{\theta^2(\partial_t^2-\Delta^2)}u\|_\infty=0$.
  \end{definition}
  
  \begin{lemma} Let $u$ be a tempered distribution on  $\mathbb{R}\times\mathbb{R}^d$ such that $u=0$ at infinity. If $\partial_t u-\Delta  u=0$, then $u=0$.
  \end{lemma} 
  \begin{proof} Take the Fourier transform (in time and space variables) of the heat equation. We find $$(i\tau+\vert\xi\vert^2) \hat u(\tau,\xi)=0.$$ Thus, $\hat u$ is supported in $\{(0,0)\}$, hence is a sum of derivatives of the Dirac mass, and $u$ is a polynomial. As $e^{\theta^2(\partial_t^2 -\Delta^2)}u$ is then a polynomial, and as $e^{\theta^2(\partial_t^2-\Delta^2)}u\in L^\infty(\mathbb{R}\times\mathbb{R}^d)$, we find that $u$ is a constant; in that case, $e^{\theta^2(\partial_t^2-\Delta^2)}u=u$. As $u=0$ at infinity, we have $u=0$.   
  \end{proof}

     \begin{lemma} Let $1<p,q<+\infty$ and $\beta\in (0, \frac 2 p+\frac d q)$. If $u\in E_{p,q,\beta}$ or  $u\in F_{p,q,\beta}$, then $u=0$ at infinity.
  \end{lemma} 
  \begin{proof} We give the proof for $u\in E_{p,q,\beta}$ (the case $u\in F_{p,q,\beta}$ is proved in a similar way). 
 Let $W(t,x)$ be the inverse Fourier transform of $e^{-\tau^2-\vert\xi\vert^4}$.  $W$ is in the Schwartz class of smooth functions with rapid decay. In particular,  $(\sqrt{\vert t\vert}+\vert\xi\vert^{d+2} \vert W(t,x)\vert \in L^\infty(\mathbb{R}\times\mathbb{R}^d)$. Writing
 \begin{equation*}\begin{split} \iint  \vert W(t,x)\vert\, \vert u(t,x)\vert \, dt \, dx&\\=
 \iint_{\sqrt \vert t\vert+ \vert x\vert<1}   \vert W(t,x)\vert\, \ \vert u(t,x)\vert \, dt \, dx +& \sum_{j=0}^{+\infty} 
 \iint_{2^j<\sqrt \vert t\vert+ \vert x\vert<2^{j+1}}   \vert W(t,x)\vert\, \ \vert u(t,x)\vert \, dt \, dx 
 \\\leq C \|u\|_{E_{p,q,\beta}}(\|\mathds{1}_{\sqrt \vert t\vert+ \vert x\vert<1}  W\|_{L^{\frac{p}{p-1}}_tL^{\frac q{q-1}}_x}&+ \sum_{j=0}^{+\infty}  2^{j(\frac 2 p+\frac d q-\beta)} \|\mathds{1}_{2^j<\sqrt \vert t\vert+ \vert x\vert<2^{j+1}}  W\|_{L^{\frac{p}{p-1}}_tL^{\frac q{q-1}}_x})
 \\ \leq C' \|u\|_{E_{p,q,\beta}}( 1+ \sum_{j=0}^{+\infty}  2^{-j \beta}      )=&C_1 \|u\|_{E_{p,q,\beta}},
 \end{split} \end{equation*} we get
 \begin{equation*}\begin{split} \vert e^{\theta^2(\partial_t^2-\Delta^2)}u(t,x)\vert \leq & \iint W(s,y) \vert u(t+\theta^2 s ,x+\theta y )\vert \, ds\, dy
 \\ \leq & C_1 \|u(t+\theta., x+\sqrt\theta .)\|_{E_{p,q,\beta}}\\=&C_1 \|u(\theta., \sqrt\theta .)\|_{E_{p,q,\beta}}\\=&C_1 \| u\|_{E_{p,q,\beta}}\theta^{-\beta/2}.
  \end{split} \end{equation*} Thus, $\lim_{\theta\rightarrow +\infty} \|e^{\theta^2(\partial_t^2-\Delta^2)}u\|_\infty=0$.
  \end{proof}
  
  The proof of Theorems \ref{theo1} and \ref{theo2} will be based on the following results on the linear heat equation:

     \begin{theorem}\label{theo3}
  Let $p,q\in (3,\infty)$ and $\beta>0$ with $\frac 2 p+\frac 3 q>\beta$. \\ a) If $  f\in E_{p/3,q/3,\beta+2}(\mathbb{R}\times\mathbb{R}^d)$, then the heat equation  
\begin{equation}\label{heat}\partial_t u=\Delta u+ f\end{equation} has a unique solution such that $u=0$ at infinity. Moreover, there exists   $C_0=C_0(p,q,\beta)>0$ such that, if denoting by $D_xu$ the gradient of $u$ and by $D^2_xu$ the Hessian of $u$ (with respect to the space variable), we  have  $$ \|  u\|_{E_{p,q,\beta}}+\| D_xu\|_{ E_{p/2,q/2,\beta+1}} +\| D^2_xu\|_{ E_{p/3,q/3,\beta+2}}\leq C_0 \|  f\|_{E_{p/3,q/3,\beta+2}} .$$
b)  I$  f\in F_{p/3,q/3,\beta+2}(\mathbb{R}\times\mathbb{R}^d)$, then the heat equation  
(\ref{heat}) has a unique solution such that $u=0$ at infinity. Moreover, there exists   $C_0=C_0(p,q,\beta)>0$ such that we  have  $$ \|  u\|_{F_{p,q,\beta}}+\| D_xu\|_{ F_{p/2,q/2,\beta+1}} +\| D^2_xu\|_{ F_{p/3,q/3,\beta+2}}\leq C_0 \|  f\|_{F_{p/3,q/3,\beta+2}} .$$
  \end{theorem}

     \begin{theorem}\label{theo4}
  Let $p,q\in (2,\infty)$ and $\beta>0$ with $\frac 2 p+\frac 3 q>\beta$.  Let $\sigma(D)$ be a Fourier multiplier (in the space variable) such that $\sigma (\xi)=  {\sigma_0(\frac \xi{\vert\xi\vert})} $ where $\sigma_0$ is a smooth function on the sphere $\mathcal{S}^{d-1}$.\\ 
  a) If $  \mathbb{F}\in E_{p/2,q/2,\beta+1}(\mathbb{R}\times\mathbb{R}^d)$, then the heat equation  
\begin{equation}\label{heat2}\partial_t u=\Delta u+ \sigma(D)\Div\mathbb{F}\end{equation} has a unique solution such that $u=0$ at infinity. Moreover, there exists   $C_0=C_0(p,q,\beta,\sigma)>0$ such that  we  have  
$$ \|  u\|_{E_{p,q,\beta}}+\| D_xu\|_{ E_{p/2,q/2,\beta+1}}  \leq C_0 \|  \mathbb{F}\|_{E_{p/2,q/2,\beta+1}} .$$
b)   If $  \mathbb{F}\in F_{p/2,q/2,\beta+1}(\mathbb{R}\times\mathbb{R}^d)$, then the heat equation  
(\ref{heat2})  has a unique solution such that $u=0$ at infinity. Moreover, there exists   $C_0=C_0(p,q,\beta,\sigma)>0$ such that  we  have  $$ \|  u\|_{F_{p,q,\beta}}+\| D_xu\|_{ F_{p/2,q/2,\beta+1}}  \leq C_0 \|  \mathbb{F}\|_{F_{p/2,q/2,\beta+1}} .$$
  \end{theorem}

    Theorem  \ref{theo3} has the following corollary on Sobolev--Morrey inequalities:
    \begin{corollary}\label{sobolev} Let $p,q\in (3,\infty)$ and $\beta>0$ with $\frac 2 p+\frac 3 q>\beta$.  Let $u=0$ at infinity.  Then  there exists   $C_0=C_0(p,q,\beta)>0$ such that we  have  $$ \|  u\|_{E_{p,q,\beta}}+\| D_xu\|_{ E_{p/2,q/2,\beta+1}} +\| D^2_xu\|_{ E_{p/3,q/3,\beta+2}}\leq C_0  (\|\partial_tu\|_{E_{p/3,q/3,\beta+2}} +\|\Delta u\|_{E_{p/3,q/3,\beta+2}})$$ and
   $$ \|  u\|_{F_{p,q,\beta}}+\| D_xu\|_{ F_{p/2,q/2,\beta+1}} +\| D^2_xu\|_{ F_{p/3,q/3,\beta+2}}\leq C_0 (\|\partial_tu\|_{F_{p/3,q/3,\beta+2}} +\|\Delta u\|_{F_{p/3,q/3,\beta+2}}). $$
    \end{corollary}
    
    Another corollary is the following theorem of Krylov \cite{KRY}:
    \begin{proposition} \label{theokryl}   Let $p,q\in (3,\infty)$ and $\beta>0$ with $\frac 2 p+\frac 3 q>\max(\beta,1)$. 
     There exists $\epsilon_0=\epsilon_0(p,q,\beta)>0$ and $C_0=C_0(p,q,\beta)>0$ such that
     \\ a)  If $  f\in E_{p/3,q/3,\beta+2}(\mathbb{R}\times\mathbb{R}^d)$, $  \vec b\in E_{p,q, 1}$, $  c\in E_{p/2,q/2, 2} $, and if  $\|  \vec b\|_{E_{p,q,1}}+\|  c\|_{E_{p/2,q/2,2}}<\epsilon_0$, then the heat equation  
\begin{equation}\label{heatdift}\partial_t u=\Delta u+ \vec b\cdot\vec\nabla u+cu+f\end{equation} has a unique solution such that $u=0$ at infinity. Moreover, we  have   \begin{equation*}\begin{split} \|  u\|_{E_{p,q,\beta}}+\| D_xu\|_{ E_{p/2,q/2,\beta+1}} &+\| D^2_xu\|_{ E_{p/3,q/3,\beta+2}}\\ \leq C_0( \|  f\|_{E_{p/3,q/3,\beta+2}}+&\|  \vec b\|_{E_{p,q,1}}+\|  c\|_{E_{p/2,q/2,2}}) .\end{split}\end{equation*}
b)    If $  f\in F_{p/3,q/3,\beta+2}(\mathbb{R}\times\mathbb{R}^d)$, $  \vec b\in F_{p,q, 1}$, $  c\in F_{p/2,q/2, 2} $, and if  $\|  \vec b\|_{F_{p,q,1}}+\|  c\|_{F_{p/2,q/2,2}}<\epsilon_0$, then the heat equation  
 (\ref{heatdift}) has a unique solution such that $u=0$ at infinity. Moreover, we  have   \begin{equation*}\begin{split} \|  u\|_{F_{p,q,\beta}}+\| D_xu\|_{ F_{p/2,q/2,\beta+1}} &+\| D^2_xu\|_{ F_{p/3,q/3,\beta+2}}\\ \leq C_0( \|  f\|_{F_{p/3,q/3,\beta+2}}+&\|  \vec b\|_{F_{p,q,1}}+\|  c\|_{F_{p/2,q/2,2}}) .\end{split}\end{equation*}
\end{proposition}

        \section{Anisotropic Hardy-Littlewood maximal function on mixed norm Lebesgue spaces}     Let  $\vec p=(p_1,\dots,p_n)\in (1,+\infty)^n$. The mixed-norm Lebesgue spaces $L^{\vec p}(\mathbb{R}^n)$ is the space of measurable functions on $\mathbb{R}^n$ such that
      $$ \|f\|_{L^{\vec p}}=\left(\int \dots \left[\int\left\{\int \vert f(x_1,\dots,x_n)\vert^{p_1}\, dx_1\right\}^{\frac {p_2}{p_1}} \right]^{\frac{p_3}{p_2}}\dots dx_n \right)^{\frac 1{p_n}}<+\infty.$$
    For $ \vec a=(a_1,\dots,a_n)\in [1,+\infty)^n$, the anisotropic cylinders $Q_{\vec a}(x,r)$ are defined as $$Q_{\vec a}(x,r)=(x_1-r^{a_1},x_1+r^{a_1})\times\dots (x_n-r^{a_n},x_n+r^{a_n})$$ and the  anisotropic Hardy-Littlewood maximal function is defined as
    $$ \mathcal{M}_{\vec a}f(x)=\sup_{r>0} \frac 1{\vert  Q_{\vec a}(x,r)\vert} \int_{Q_{\vec a}(x,r)}\vert f(y)\vert\, dy.$$
    
We then have the following boundedness result for the   
  anisotropic Hardy-Littlewood maximal function   on mixed norm Lebesgue spaces \cite{HUA1, HUA2}:
  
  \begin{proposition}\label{anisotrop}  For $\vec p=(p_1,\dots,p_n)\in (1,+\infty)^n$ and  $ \vec a=(a_1,\dots,a_n)\in [1,+\infty)^n$, there exists a constant $C=C(\vec p,\vec a)$ such that
  $$ \|  \mathcal{M}_{\vec a}f\|_{L^{\vec p}} \leq C \|  f\|_{L^{\vec p}}.$$
  \end{proposition}
  
   \section{Harmonic analysis on the parabolic space}
   We endow $X=\mathbb{R}\times\mathbb{R}^d $ with the Lebesgue measure $d\mu(t,x)=dt\, dx $ and the parabolic metric $\rho((t,x),(s,y))=\sqrt{\vert t-s\vert}+\vert x-y\vert$. The parabolic ball $B_r(t,x)$ and the parabolic cylinder $C_r(t,x)$ are defined as
   $$ B_r(t,x)=\{(s,y)\ /\ \rho((t,x),(s,y))<r\}\text{ and } C_r(t,x)=(t-r^2,t+r^2)\times B(x,r).$$
   We have
   $$  B_r(t,x)\subset  C_r(t,x)\subset  B_{2r}(t,x).$$
   $(X,\rho,\mu)$ is a space of homogeneous type \cite{COI}   of homogeneous  dimension $Q=d+2$:
   $$ \mu(B_r(t,x))=\mu(B_1(0,0)) r^Q.$$ We associate to this space three useful operators on (non-negative) functions: the Hardy--Littlewood maximal function $\mathcal{M}_f$ and the Riesz potentials $\mathcal{I}_\alpha f$  (where $0<\alpha<Q$) defined by
   $$\mathcal{M}_f(t,x) =\sup_{r>0} \frac 1{\mu(B_r(t,x))}\int_{B_r(t,x)} \vert f(s,y)\vert \,d\mu(s,y)$$ and
     $$\mathcal{I}_\alpha f(t,x)=\int_X \frac 1{\rho((t,x),(s,y))^{Q-\alpha}} f(s,y)\ d\mu(s,y).$$
   For $1<p\leq q$, we define the Morrey space $\mathcal{M}^{p,q}_2$ by the space of measurable functions such that
   $$ \| f\|_{\mathcal{M}^{p,q}_2}=\sup_{r>0, (t,x)\in X}  \mu(B_r(t,x))^{\frac 1 q-\frac 1 p} \|\mathds{1}_{B_r(t,x)} \|_p <+\infty.$$
   We have $L^q(\mathbb{R}\times\mathbb{R}^d)\subset \mathcal{M}^{p,q}_2$, $E_{p,p,\beta}=\mathcal{M}^{p,\frac{d+2}\beta}$ and $E_{p,q,\beta}\subset \mathcal{M}^{\min(p,q), \frac{d+2}\beta}_2$.
   
   The main tool we shall use is Hedberg's inequality \cite{HED}:
   
   \begin{proposition}\label{hedb} Let $1<p\leq q<+\infty$ and $0<\alpha<\frac Q q$. Then there exists a constant $C=C(d,p,q,\alpha)$ such that, for every $f\in \mathcal{M}^{p,q}_2$, we have
   $$ \mathcal{I}_\alpha f(t,x)\leq C (\mathcal{M}_f(t,x))^{1-\frac{\alpha q}Q} \|f\|_{\mathcal{M}^{p,q}_2}^{\frac{\alpha q}Q}.$$
   \end{proposition}
   \begin{proof} We easily check that, for every $R>0$, we have 
   $$ \int_{B_R(t,x)} \frac 1{\rho((t,x),(s,y))^{Q-\alpha}} \vert f(s,y)\vert\ d\mu(s,y)\leq C R^\alpha \mathcal{M}_f(t,x)$$ and
   $$ \int_{X\setminus B_R(t,x)} \frac 1{\rho((t,x),(s,y))^{Q-\alpha}} \vert f(s,y)\vert\ d\mu(s,y)\leq C R^{\alpha-\frac Q q}  \|f\|_{\mathcal{M}^{p,q}_2}.$$ We then take \begin{equation*} R^{\frac Q q}=\frac{\|f\|_{\mathcal{M}^{p,q}_2}} {\mathcal{M}_f(t,x)}.\tag*{\qedhere}\end{equation*}
   \end{proof}
   
   Applying Proposition \ref{hedb} to Krylov spaces, we obtain:
     \begin{proposition}\label{prop4} Let $1<p,q<+\infty$ and $0<\beta<\frac 2 p+\frac d q$. Then 
     \\ a) $f\mapsto \mathcal{M}_f$ is bounded from $E_{p,q,\beta}$ to $E_{p,q,\beta}$ and from $F_{p,q,\beta}$ to $F_{p,q,\beta}$.
     \\ b) If $\beta>1$,  $f\mapsto \mathcal{I}_1f$  is bounded from $E_{p,q,\beta}$ to $E_{\frac \beta{\beta-1}p,\frac \beta{\beta-1}q,\beta-1}$ and from $F_{p,q,\beta}$ to $F_{\frac \beta{\beta-1}p,\frac \beta{\beta-1}q,\beta-1}$.
  \\ c) If $\beta>2$,  $f\mapsto \mathcal{I}_2f$  is bounded from $E_{p,q,\beta}$ to $E_{\frac \beta{\beta-2}p,\frac \beta{\beta-2}q,\beta-2}$ and from $F_{p,q,\beta}$ to $F_{\frac \beta{\beta-2}p,\frac \beta{\beta-2}q,\beta-2}$
     \end{proposition}
     
     \begin{proof}  We consider only the case of $E_{p,q,\beta}$, as the proof for $F_{p,q,\beta}$ is similar. Let us estimate $\mathcal{M}_f$ on a cylinder $C_r(t,x)$. We have $\mathcal{M}_f\leq \mathcal{M}_{f_1}+\mathcal{M}_{f_2}$, where $f_1=\mathds{1}_{C_{4r}(t,x)} f$ and $f_2=f-f_1$. By Proposition \ref{anisotrop}, we know that
     $$ \|\mathcal{M}_{f_1}\|_{L^p_tL^q_x}\leq C \|{f_1}\|_{L^p_tL^q_x}\leq C' \|f\|_{E_{p,q,\beta}}r^{\frac 2 p+\frac d q-\beta}. $$ On the other hand, for $(s,z)\in C_r(t,x)\subset B_{2r}(t,x)$, since $B_{4r}(t,x)\subset C_{4r}(t,x)$, $$\mathcal{M}_{f_2}(s,z) \leq \sup_{\rho>2r} \frac 1{\vert B_\rho(s,z)\vert} \iint_{B_\rho(s,z)} \vert f(\sigma,y)\vert\, d\sigma\, dy\leq C' \|f\|_{E_{p,q,\beta}} r^{-\beta}$$
     so that
      $$ \|\mathds{1}_{C_r(t,x)}\mathcal{M}_{f_2}\|_{L^p_tL^q_x}\leq C  \|f\|_{E_{p,q,\beta}}r^{\frac 2 p+\frac d q-\beta}. $$ a) is proved. b) and c) are then direct consequences of Hedberg's inequality (Proposition \ref{hedb}).
     \end{proof}
     
     \section{The heat equation on $\mathbb{R}\times\mathbb{R}^d$.}
     
     In this section, we solve the  heat equation on $\mathbb{R}\times\mathbb{R}^d$:
      \begin{equation}\label{cauchy}\left\{
 \begin{split} &\partial_t   u =\Delta  u+ g
\\&   u =  0\text{ at infinity}
\end{split}\right.\end{equation}
where $g=f\in E_{p,q,\beta}$ or in $F_{p,q,\beta}$ with $2<\beta<\frac 2 p+\frac d q$ or $g=\sigma(D)\Div \mathbb{F}$ where $\mathbb{F}\in E_{p,q,\beta}$ or in $F_{p,q,\beta}$ with $1<\beta<\frac 2 p+\frac d q$. The solution $u$ of equation (\ref{cauchy}) is given by the Duhamel formula
$$ u=\int_{-\infty}^t e^{(t-s)\Delta} g(s,.)\, ds.$$ In order to estimate $u$ and its derivatives, we need some estimates on the size of the kernel of $e^{t\Delta}$ and its derivatives, or on the derivatives of $\sigma(D)e^{t\Delta}$.

     \begin{lemma} \label{size} Let $\psi\in \mathcal{S}(\mathbb{R}^d)$ and, for $\theta>0$,  $\psi_\theta(x)=\frac 1{\theta^d} \psi(\frac x \theta)$.  Let $\sigma(D)$ be a Fourier multiplier (in the space variable) such that $\sigma (\xi)=  {\sigma_0(\frac \xi{\vert\xi\vert})} $ where $\sigma_0$ is a smooth function on the sphere $\mathcal{S}^{d-1}$. Then, for $\alpha\in \mathbb{N}^d$,
     $$ \vert \partial_\alpha \sigma(D)\psi_\theta(x)\vert \leq C_{\alpha,\psi,\sigma} \frac 1{(\theta+\vert x\vert)^{d+\vert\alpha\vert}}.$$
      \end{lemma}
     
     \begin{proof} We have
   $$  \| \partial_\alpha \sigma(D)\psi_\theta\|_\infty =\theta^{-d-\vert\alpha\vert}  \| \partial_\alpha \sigma(D)\psi\|_\infty$$ and 
   $$  \|  \vert x\vert^{d+\vert\alpha\vert} \partial_\alpha \sigma(D)\psi_\theta\|_\infty =\| \vert x\vert^{d+\vert\alpha\vert}    \partial_\alpha \sigma(D)\psi\|_\infty$$ Thus, we may assume $\theta=1$. 
   We have
   $$ \vert  \partial_\alpha \sigma(D)\psi_\theta(x)\vert\leq \frac 1{(2\pi)^d} \|\sigma_0\|_\infty \int \vert \xi\vert^{\vert\alpha\vert} \vert\hat\psi(\xi)\vert\, dx.$$ On the other hand, taking a smooth function $\phi$ such that $\phi(\xi)=1$ for $\vert\xi\vert<1$ and $\phi(\xi)=0$ for $\vert\xi\vert>2$, we have for every $R>0$, for $1\leq k\leq d$,     
   \begin{equation*}\begin{split}
   \vert x_k\vert^{d+\vert\alpha\vert} \vert \partial_\alpha  \phi(\frac D R)\sigma(D)\psi_\theta(x)\vert\leq &  C (R\vert x\vert)^{d+\alpha} \|\sigma_0\|_\infty \|\hat \psi\|_\infty 
   \end{split}\end{equation*}
   and
   \begin{equation*}\begin{split}
\vert x\vert^2   \vert x_k\vert^{d+\vert\alpha\vert} \vert \partial_\alpha(1-\phi(\frac D R)) \sigma(D)\psi_\theta(x)\vert  & \\\leq \frac 1{(2\pi)^d} \int \vert \Delta\partial_k^{d+\vert\alpha\vert}& \left((1-\phi(\frac\xi R))\xi^\alpha \sigma(\xi) \hat\psi(\xi)\right)\vert\, d\xi
\\ \leq C\int_{\vert\xi\vert>R} \frac{d\xi}{\vert \xi\vert^{d+2}}&= C'\frac 1 {R^2}
   \end{split}\end{equation*}
   (as $\vert \partial^\beta(\hat \psi(\xi))\vert \leq C_{\beta,\psi} \vert \xi\vert^{-\vert\beta\vert}$, $\vert \partial^\beta( \xi^\alpha \sigma(\xi))\vert \leq C_{\alpha,\beta,\sigma} \vert \xi\vert^{\vert\alpha\vert-\vert\beta\vert}$ and  $\vert \partial^\beta( 1-\phi(\frac\xi R)))\vert \leq C_{\beta,\phi} \vert \xi\vert^{-\vert\beta\vert}$). Taking $R=\frac 1{\vert x\vert}$, we get $x_k^{d+\vert\alpha\vert}\partial^\alpha\sigma(D)\psi \in L^\infty$.
     \end{proof}
     
     A useful result on the heat kernel is its maximal regularity in  $L^p_t L^q_x$ or in $L^q_xL^p_t$:
     \begin{proposition}\label{regmax} Let $1<p,q<+\infty$.  Let $\sigma(D)$ be a Fourier multiplier (in the space variable) such that $\sigma (\xi)=  {\sigma_0(\frac \xi{\vert\xi\vert})} $ where $\sigma_0$ is a smooth function on the sphere $\mathcal{S}^{d-1}$.
     \\ a)  If $h\in L^p_t L^q_x$, then, for $1\leq i,j\leq d$, 
     $$\| \int_{-\infty}^t e^{(t-s)\Delta} \sigma(D)\partial_i\partial_j h\, ds\|_{L^p_t L^q_x}\leq C \|h\|_{L^p_t L^q_x}.$$
      a)  If $h\in  L^q_xL^p_t$, then, for $1\leq i,j\leq d$, 
     $$\| \int_{\infty}^t e^{(t-s)\Delta} \sigma(D)\partial_i\partial_j h\, ds\|_{L^q_xL^p_t}\leq C \|h\|_{L^q_xL^p_t}.$$
     \end{proposition}
     \begin{proof} This is a classical result, obtained through the theory of vector valued singular integrals \cite{BCP}. We sketch the proof given in \cite{PGL1}. Let $W(x)=\frac 1{(4\pi)^{d/2}} e^{-\frac{\vert x\vert^2}4}$
 and let $$\Omega(t,x)=   \frac 1{t^{\frac{d+2}2}} (\sigma(D)\partial_i\partial_j W)(\frac x {\sqrt t})\text{ for }t>0, \quad =0  \text{ for }t<0.$$
 
 We then consider the operator
\begin{equation}\label{singint} T(g)(t,x)=\int_\mathbb{R}\int_{\mathbb{R}^d}   \Omega(t-s,x-y) g(s,y)\, ds\, dy.\end{equation} We have $$(\partial_t-\Delta)T( g)=\sigma(D)\partial_i\partial_j g$$ and, taking the Fourier transform,
 $$\widehat{T(g)}(\tau,\xi)=-\frac{\xi_i\xi_j}{i\tau+\vert \xi\vert^2}\sigma(\xi) \hat g(\tau,\xi).$$ Thus, $T$ is bounded on $L^2(\mathbb{R}\times\mathbb{R}^d)$.   $T$ can be seen as Calder\' on--Zygmund operator on the parabolic space $\mathbb{R}\times\mathbb{R}^d$ (for the parabolic distance): from Lemma \ref{size}, we see that $\vert\Omega(t,x)\vert\leq C \frac 1{(\sqrt{\vert t\vert}+\vert x\vert)^{d+2}}$,  $\vert\partial_i\Omega(t,x)\vert\leq C \frac 1{(\sqrt{\vert t\vert}+\vert x\vert)^{d+3}}$,  and  $\vert\partial_t\Omega(t,x)\vert=\vert\Delta\Omega(t,x)\vert\leq C \frac 1{(\sqrt{\vert t\vert}+\vert x\vert)^{d+4}}$.  Thus, we get that  $T$ is bounded on $L^p(\mathbb{R}\times\mathbb{R}^d)$ for $1<p<+\infty$. Let $T_t$ be the operator
 $$T_t(v)=\int_{\mathbb{R}^d}  \frac 1 {t} \Omega(t,x-y) v(y)\,  dy  $$ and $T_x$ be the operator
 $$T_x(w)=\int_{\mathbb{R}}   \frac 1 {t-s} \Omega(t-s,x)  w(s)\, ds$$ We have  $$\partial_tT_t(v)=\int_{\mathbb{R}^d}  \frac 1 {t} \Delta \Omega(t,x-y) v(y)\,  dy  $$ and 
 $$\partial_kT_x(w)=\int_{\mathbb{R}}   \frac 1 {t-s} \partial_k\Omega(t-s,x)  w(s)\, ds.$$  From  Lemma \ref{size}, we find $$ \|T_t\|_{L^q_x\mapsto L^q_x}\leq C\frac 1 t \text{ and } \|\partial_tT_t\|_{L^q_x\mapsto L^q_x}\leq C\frac 1 {t^2}$$ so that the continuity $L^q_t(L^q_x)\mapsto L^q_t(L^q_x)$ of $T$ can be extended to the continuity $L^p_t(L^q_x)\mapsto L^p_t(L^q_x)$. Similarly, we have  $$ \|T_x\|_{L^p_t\mapsto L^p_t}\leq C\frac 1 {\vert x\vert^d} \text{ and } \|\partial_kT_x\|_{L^p_t\mapsto L^p_t}\leq C\frac 1 {\vert x\vert^{d+1}}  $$ and the continuity $L^p_x(L^p_t)\mapsto L^p_x(L^p_t)$ of $T$ can be extended to the continuity $L^q_x(L^p_t)\mapsto L^q_x(L^p_t)$.
  \end{proof}

  Similar estimates hold for Krylov spaces:
   \begin{proposition}\label{regmaxkryl} Let $1<p,q<+\infty$ and $0<\beta<\frac 2 p+\frac d q$.  Let $\sigma(D)$ be a Fourier multiplier (in the space variable) such that $\sigma (\xi)=  {\sigma_0(\frac \xi{\vert\xi\vert})} $ where $\sigma_0$ is a smooth function on the sphere $\mathcal{S}^{d-1}$.
     \\ a)  If $h\in E_{p,q,\beta}$, then, for $1\leq i,j\leq d$, 
     $$\| \int_{-\infty}^t e^{(t-s)\Delta} \sigma(D)\partial_i\partial_j h\, ds\|_{E_{p,q,\beta}}\leq C \|h\|_{E_{p,q,\beta}}.$$
      a)  If $h\in  F_{p,q,\beta}$, then, for $1\leq i,j\leq d$, 
     $$\| \int_{-\infty}^t e^{(t-s)\Delta} \sigma(D)\partial_i\partial_j h\, ds\|_{F_{p,q,\beta}}\leq C \|h\|_{F_{p,q,\beta}}.$$
     \end{proposition}

     \begin{proof}  We consider only the case of $E_{p,q,\beta}$, as the proof for $F_{p,q,\beta}$ is similar. Let us estimate $T(h)$ on a cylinder $C_r(t,x)$ (where $T$ is the operator given by (\ref{singint})). We have $T(h)=T(h_1)+T(h_2)$, where $h_1=\mathds{1}_{C_{4r}(t,x)} h$ and $h_2=h-h_1$. By Proposition \ref{regmax}, we know that
     $$ \|{h_1}\|_{L^p_tL^q_x}\leq C \|{h_1}\|_{L^p_tL^q_x}\leq C' \|h\|_{E_{p,q,\beta}}r^{\frac 2 p+\frac d q-\beta}. $$ On the other hand, for $(s,z)\in C_r(t,x)\subset B_{2r}(t,x)$, since $B_{4r}(t,x)\subset C_{4r}(t,x)$,
     \begin{equation*}\begin{split}\vert{h_2}(s,z) \vert\leq & C\iint_{\rho((\sigma,y),(t,x))>4r} \frac 1{(\sqrt{ \vert\tau-s\vert}+\vert z-y\vert)^{d+2}} \vert h(\tau,y)\vert\, d\tau\, dy
   \\  \leq & C'\iint_{\rho((\sigma,y),(t,x))>4r} \frac 1{(\sqrt{ \vert\tau-t\vert}+\vert x-y\vert)^{d+2}} \vert h(\tau,y)\vert\, d\tau\, dy
  \\ \leq & C''\sum_{j=0}^{+\infty}   \frac 1{ (4^jr)^{d+2}} \iint_{C_{4^{j+1}r}} \vert h(\tau,y)\vert\, d\tau\, dy
     \\\leq &C''' \|h\|_{E_{p,q,\beta}} r^{-\beta}
     \end{split}\end{equation*}
     so that
  \begin{equation*} \|\mathds{1}_{C_r(t,x)}  {h_2}\|_{L^p_tL^q_x}\leq C  \|\|_{E_{p,q,\beta}}r^{\frac 2 p+\frac d q-\beta}. \tag*{\qedhere}\end{equation*}
     \end{proof}

           We may now easily prove Theorems \ref{theo3} and \ref{theo4}, Corollary \ref{sobolev} and Proposition \ref{theokryl}. Again, we shall    consider only the case of $E_{p,q,\beta}$, as the proofs for $F_{p,q,\beta}$ are similar.
           
           \subsection*{Proof of Theorem \ref{theo3}}
 \begin{proof}          The solution $u$ of equation $\partial_t u=\Delta u+ f$ is given by $u=\int_{-\infty}^t e^{t-s)\Delta} f(s,.)\, ds$.  By proposition (\ref{regmaxkryl}), we already know that
 $$\| D^2_xu\|_{ E_{p/3,q/3,\beta+2}}\leq C \|  f\|_{E_{p/3,q/3,\beta+2}} .$$ On the other hand, from Lemma \ref{size}, we see that $$\vert D_xu(t,x)\vert\leq C \iint \frac 1{ (\sqrt{\vert t-s\vert}+\vert x-y\vert)^{d+1}} \vert f(s,y)\vert\, ds\, dy= C \mathcal{I}_1(\vert f\vert)(t,x)$$ and $$\vert u(t,x)\vert\leq C \iint \frac 1{ (\sqrt{\vert t-s\vert}+\vert x-y\vert)^{d}} \vert f(s,y)\vert\, ds\, dy=C \mathcal{I}_2(\vert f\vert)(t,x).$$  We then apply Proposition \ref{prop4} to get
      $$ \| D_xu\|_{ E_{p/2,q/2,\beta+1}}  \leq C  \|  f\|_{E_{p/3,q/3,\beta+2}} $$
 and      \begin{equation*} \|  u\|_{E_{p,q,\beta}} \leq C \|  f\|_{E_{p/3,q/3,\beta+2}} .\tag*{\qedhere}\end{equation*}
  \end{proof}

             \subsection*{Proof of Theorem \ref{theo4}}
 \begin{proof}          The solution $u$ of equation $\partial_t u=\Delta u+ \sigma(D)\Div\mathbb{F}$ is given by $u=\int_{-\infty}^t e^{t-s)\Delta} \sigma(D)\Div\mathbb{F}(s,.)\, ds$.  By proposition (\ref{regmaxkryl}), we already know that
 $$\\| D_xu\|_{ E_{p/2,q/2,\beta+1}}  \leq C_0 \|  \mathbb{F}\|_{E_{p/2,q/2,\beta+1}}.$$ On the other hand, from Lemma \ref{size}, we see that $$\vert  u(t,x)\vert\leq C \iint \frac 1{ (\sqrt{\vert t-s\vert}+\vert x-y\vert)^{d+1}} \vert \mathbb{F}(s,y)\vert\, ds\, dy= C \mathcal{I}_1(\vert \mathbb{F}\vert)(t,x).$$  We then apply Proposition \ref{prop4} to get
        \begin{equation*} \|  u\|_{E_{p,q,\beta}} \leq C \|  \mathbb{F}\|_{E_{p/2,q/2,\beta+1}}  .\tag*{\qedhere}\end{equation*}
  \end{proof}

           \subsection*{Proof of Corollary \ref{sobolev}}
\begin{proof} We just write $\partial_tu-\Delta u=f$ with $f=\partial_tu-\Delta u$. By Theorem \ref{theo3}, we have   $$ \|  u\|_{E_{p,q,\beta}}+\| D_xu\|_{ E_{p/2,q/2,\beta+1}} +\| D^2_xu\|_{ E_{p/3,q/3,\beta+2}}\leq C_0  (\|\partial_tu\|_{E_{p/3,q/3,\beta+2}} +\|\Delta u\|_{E_{p/3,q/3,\beta+2}})$$ \end{proof}
     
               \subsection*{Proof of Proposition \ref{theokryl}}
\begin{proof} Let $X=\{u\in   E_{p,q,\beta}\ /\ D_xu\in { E_{p/2,q/2,\beta+1}},  D^2_xu\in { E_{p/3,q/3,\beta+2}}\}$, normed with
$$\|u\|_X= \|  u\|_{E_{p,q,\beta}}+\| D_xu\|_{ E_{p/2,q/2,\beta+1}} +\| D^2_xu\|_{ E_{p/3,q/3,\beta+2}}.$$ We are looking for a solution in $X$ of the equation
$$ u= T(f)+ T(\vec b\cdot\vec\nabla u)+T(cu)$$, where $T(f)=\int_{-\infty}^t e^{(t-s)\Delta} f(s,.)\, ds$. By Theorem \ref{theo3}, we have $$\|T(f)\|_X\leq C_1  \|  f\|_{E_{p/3,q/3,\beta+2}},$$
$$\|T(\vec b\cdot\vec\nabla u)\|_X\leq C_2  \|  \vec b\cdot\vec\nabla u\|_{E_{p/3,q/3,\beta+2}}\leq C_2 \|\vec b\|_{E_{p,q,1}} \|D_xu\|_{ E_{p/2,q/2,\beta+1}},$$
and $$\|T(cu)\|_X\leq C_3  \|  cu\|_{E_{p/3,q/3,\beta+2}}\leq C_3 \|c\|_{E_{p/2,q/2,2}} \|u\|_{p,q,\beta}.$$
For $C_2 \|\vec b\|_{E_{p,q,1}}+ C_3 \|c\|_{E_{p/2,q/2,2}} <\frac 1 2$, the operator $S(u)=u-T(\vec b\cdot\vec\nabla u)-T(cu)$ is an isomorphism of $X$, with $\|S^{-1}\|_{X\mapsto X}\leq 2$. The solution $u$ is then given by $u=S^{-1}(T(f))$, with $\|u\|_X\leq 2C_1  \|  f\|_{E_{p/3,q/3,\beta+2}}$.
\end{proof}

          \section{The Navier--Stokes equations.}
          We prove here Theorem \ref{theo1}. The proof of Theorem \ref{theo2} is similar. Again, we  consider only the case of $E_{p,q,\beta}$.

          \subsection*{Proof of Theorem \ref{theo1}}
\begin{proof}The Navier--Stokes equations we study are  
 \begin{equation}\label{NSE2}\left\{
 \begin{split} &\partial_t \vec u =\Delta\vec u-\vec \nabla p-\vec u\cdot \vec \nabla\vec u + \vec f+\Div\mathbb{F}
\\& \Div\vec u=0
\\& \vec u(0,.)=  0
\end{split}\right.\end{equation}
where $\vec f\in E_{p/3,q/3,3}$ with $\Div\vec f=0$  and $\mathbb{F}\in E_{p/2,q/2,2}$. We rewrite (\ref{NSE2}) as
$$ \partial_t\vec u-\Delta\vec u=\vec f+({\rm Id}-\frac1 \Delta \vec\nabla\Div)\Div(\mathbb{F}-\vec u\otimes\vec u).$$  (${\rm Id}-\frac1 \Delta \vec\nabla\Div=\mathbb{P}$ is the Leray projection operator on solenoidal vector fields.)
 
  Let $X=\{\vec u\in   E_{p,q,1}\ /\ D_x\vec u\in { E_{p/2,q/2, 2}} \}$, normed with
$$\|u\|_X= \|  \vec u\|_{E_{p,q,1}}+\| D_x\vec u\|_{ E_{p/2,q/2, 2}}.$$ We are looking for a solution in $X$ of the equation
$$\vec  u= \vec T(\vec f)+ \vec T(\mathbb{P}\Div\mathbb{F})-\vec T(\mathbb{P}\Div(\vec u\otimes\vec u)),$$ where $\vec T(\vec f)= \int_0^t e^{(t-s)\Delta} \vec f(s,.)\, ds$ ($=\int_{-\infty}^t e^{(t-s)\Delta} \vec f(s,.)\, ds$  if we extend $\vec f$ to $(-\infty,0)$ with $\vec f(t,x)=0$ for $t<0$). By Theorem \ref{theo3}, we have $$\|\vec T(\vec f)\|_X\leq C_1  \| \vec  f\|_{E_{p/3,q/3, 3}}.$$ By Theorem \ref{theo4}, we have
$$\|\vec T(\mathbb{P}\Div\mathbb{F})\|_X\leq C_2  \|  \mathbb{F}\|_{E_{p/2,q/2,2}} $$
and $$\|\vec T(\mathbb{P}\Div(\vec u\otimes\vec u))\|_X\leq C_3  \|  \vec u\otimes\vec u\|_{E_{p/2,q/2,2}}\leq C_3  \|\vec u\|_{p,q,1}^2.$$
For $C_1  \| \vec  f\|_{E_{p/3,q/3, 3}}+C_2  \|  \mathbb{F}\|_{E_{p/2,q/2,2}}<\frac 1{4C_3}$, we find a solution $\vec u\in X$ with $\|\vec u\|_X\leq 2(C_1  \| \vec  f\|_{E_{p/3,q/3, 3}}+C_2  \|  \mathbb{F}\|_{E_{p/2,q/2,2}})$.   \end{proof}

       \section*{Appendix: Parabolic Besov spaces}
   In this appendix, we describe some results on the heat equation on parabolic Besov spaces that may shed a new light on  many of the inequalities we obtained throughout the paper, which are to be viewed more as regularity assertions than as results on existence of solutions.  
   
   For $\delta>0$, let  us define the parabolic Besov space
    $ \mathcal{B}^{-\delta}_{[2],\infty,\infty}$ as the space of tempered distributions on $\mathbb{R}\times\mathbb{R}^d$ such that
    $$ \|u\|_{ \mathcal{B}^{-\delta}_{[2],\infty,\infty}}=\sup_{\theta>0}\theta^{\delta/2} \|e^{\theta^2(\partial_t^2-\Delta^2)}u(t,x) \|_{L^\infty(\mathbb{R}\times\mathbb{R}^d)}<+\infty.$$ (In particular, $E_{p,q,\beta}\subset  \mathcal{B}^{-\beta}_{[2],\infty,\infty}$ and $F_{p,q,\beta}\subset  \mathcal{B}^{-\beta}_{[2],\infty,\infty}$.)
    
    It is easy to check that $\mathcal{B}^{-\beta}_{[2],\infty,\infty}$ coincides with the (realization of the) homogeneous Besov space defined through the parabolic Littlewood-Paley decomposition  (for instance, see Triebel's book \cite{TRI}). Let $\phi(\tau,\xi)$ be a smooth function equal to $1$ when $\sqrt{\vert\tau\vert}+\vert\xi\vert<1$ and to $0$ when $\sqrt{\vert\tau\vert}+\vert\xi\vert>2$. Let $S_jf$ be defined as the inverse Fourier transform of $\phi(\frac\tau{4^j},\frac\xi{2^j}) \hat f(\tau,\xi)$. Then $\mathcal{B}^{-\beta}_{[2],\infty,\infty}$ is the Banach space of tempered distributions $f$ such that $\sup_{j\in\mathbb{Z}} 2^{j\delta} \|S_jf\|_\infty<+\infty$.

We have the following Gagliardo--Nirenberg inequality
\begin{proposition} \label{gagl}Let $\delta>0$, $1<p,q<+\infty$ and $\beta\in (0, \frac 2 p+\frac d q)$.\\ a)  If $u\in \mathcal{B}^{-\delta}_{[2],\infty,\infty}$ and $\partial_t u, \Delta u\in E_{p,q,\beta}$, then 
$$ \| D_xu\|_{E_{\frac{2+\delta}{1+\delta}p,\frac{2+\delta}{1+\delta}q,\frac{1+\delta}{2+\delta}\beta}} \leq C  \|u\|_{ \mathcal{B}^{-\delta}_{[2],\infty,\infty}}^{\frac{1}{2+\delta}} (\| \partial_t u\|_{E_{p,q,\beta}}+\| \Delta u\|_{E_{p,q,\beta}})^{\frac{1+\delta}{2+\delta}}.$$
b) If $u\in \mathcal{B}^{-\delta}_{[2],\infty,\infty}$ and $\partial_t u, \Delta u\in F_{p,q,\beta}$, then 
$$ \| D_xu\|_{F_{\frac{2+\delta}{1+\delta}p,\frac{2+\delta}{1+\delta}q,\frac{1+\delta}{2+\delta}\beta}} \leq C  \|u\|_{ \mathcal{B}^{-\delta}_{[2],\infty,\infty}}^{\frac{1}{2+\delta}} (\| \partial_t u\|_{F_{p,q,\beta}}+\| \Delta u\|_{F_{p,q,\beta}})^{\frac{1+\delta}{2+\delta}}.$$
\end{proposition}
    
    Our final resuts extend  Corollary \ref{sobolev} and Theorem \ref{theo3} to parabolic Besov spaces:
    
        \begin{theorem}\label{theo5}
  Let $\delta>0$.   Let $u=0$ at infinity.  Then $u\in {\mathcal{B}^{-\delta}_{[2],\infty,\infty}}$ if and only if $  \partial_tu\in \mathcal{B}^{-\delta-2}_{[2],\infty,\infty}$ and $  \Delta u\in \mathcal{B}^{-\delta-2}_{[2],\infty,\infty}$  Moreover, the   norms     $ \|  u\|_{\mathcal{B}^{-\delta}_{[2],\infty,\infty}}$ and  $ \|  \partial_tu\|_{\mathcal{B}^{-\delta-2}_{[2],\infty,\infty}}+ \|  \Delta u\|_{\mathcal{B}^{-\delta-2}_{[2],\infty,\infty}} $ are equivalent. 
   \end{theorem}

     \begin{theorem}\label{theo6}
  Let  $\delta>2$. If $  f\in  \mathcal{B}^{-\delta}_{[2],\infty,\infty}$, then the heat equation  
\begin{equation}\label{heat6}\partial_t u=\Delta u+ f\end{equation} has a unique solution such that $u=0$ at infinity. Moreover, $u\in \mathcal{B}^{-\delta+2}_{[2],\infty,\infty} $ and  $$ \|  u\|_{\mathcal{B}^{-\delta+2}_{[2],\infty,\infty}} \leq C  \|  f\|_{\mathcal{B}^{-\delta}_{[2],\infty,\infty}} .$$
   \end{theorem}
   
   Those results are not very new, as they can be recovered from the voluminous literature on anisotropic Besov spaces We present them for sake of completeness.
   
   \subsection*{Proof of Proposition \ref{gagl}.}
   \begin{proof} As $u=0$ at infinity, we may write
   $$ u=\int_0^{\infty} (-\partial_\theta) (e^{\theta^2(\partial_t^2-\Delta^2)}u)\, d\theta$$ and
   $$ D_xu=\int_0^{\infty} (-\partial_\theta) D_x (e^{\theta^2(\partial_t^2-\Delta^2)}u)\, d\theta.$$
  For $R>0$, we rewrite this as
   \begin{equation*}\begin{split}D_xu=  &-\int_0^{R}  (2\theta \partial_t  e^{\theta^2 \partial_t^2}\sqrt\theta D_x e^{- \theta^2 \Delta^2}) \partial_t u\, \frac{d\theta}{\sqrt\theta}
   \\ -&\int_R^{\infty}  (2\theta^2 \partial_t^2  e^{\frac 34 \theta^2 \partial_t^2}\sqrt\theta D_x e^{-\frac34\theta^2 \Delta^2}) (e^{\frac{\theta^2}4(\partial_t^2-\Delta^2) }u)\, \frac{d\theta}{ \theta^{3/2}}
   \\+&\int_0^{R}  ( 2 e^{\theta^2 \partial_t^2} \theta^{3/2} D_x\Delta e^{- \theta^2  \Delta^2} )  \Delta u\, \frac{d\theta}{\sqrt\theta}
   \\+&\int_R^{\infty}  ( 2 e^{\frac 34 \theta^2 \partial_t^2} \theta^{5/2} D_x\Delta^2  e^{-\frac34\theta^2\Delta^2} ) (e^{\frac{\theta^2}4(\partial_t^2-\Delta^2)}u)\, \frac{d\theta}{ \theta^{3/2}} \\=& I+II+III+IV.\end{split}\end{equation*} 
   
   Recall that $\mathcal{M}_f$ designs the parabolic Hardy--Littlewood maximal function. If $\psi\in \mathcal{S}(\mathbb{R}\times\mathbb{R}^d)$, we have, for every $\theta>0$,
   $$ \iint \frac 1{\theta^{\frac{d+2}2}} \vert \psi(\frac s\theta, \frac y{\sqrt\theta}) \vert \vert f(t-s,x-y)\, ds\, dy\leq C_\psi \mathcal{M}_f(t,x)$$ and
   $$ \iint \frac 1{\theta^{\frac{d+2}2}} \vert \psi(\frac s\theta, \frac y{\sqrt\theta}) \vert \vert f(t-s,x-y)\, ds\, dy\leq \|\psi\|_1 \|f\|_\infty.$$
   Thus, we have
   $$ \vert I\vert\leq C \sqrt R    \mathcal{M}_{\partial_t u}\text{ and } \vert III\vert\leq C \sqrt R    \mathcal{M}_{\Delta u}$$ and
   $$ \vert II\vert+\vert IV\vert \leq C R^{-\frac {1+\delta}2} \|u\|_{\mathcal{B}^{-\delta}_{[2],\infty,\infty}}.$$
   Taking $\sqrt R=\left( \frac {\|u\|_{\mathcal{B}^{-\delta}_{[2],\infty,\infty}}}{\mathcal{M}_{\partial_t u}(t,x)+ \mathcal{M}_{\Delta u}(t,x)}\right)^{\frac 1{2+\delta}}$, we get 
   $$ \vert D_xu(t,x)\vert \leq C \|u\|_{\mathcal{B}^{-\delta}_{[2],\infty,\infty}}^{\frac 1{2+\delta}} {(\mathcal{M}_{\partial_t u}(t,x)+ \mathcal{M}_{\Delta u}(t,x))}^{\frac {1+\delta}{2+\delta}}$$ and we conclude by Proposition \ref{prop4} a).
   \end{proof}

   \subsection*{Proof of Theorem \ref{theo5}.}
\begin{proof} Writing
 $$  e^{\theta^2(\partial_t^2-\Delta^2)}(\partial_t u)= \frac 1 \theta (\theta\partial_te^{\frac 3 4\theta^2(\partial_t^2-\Delta^2)}) (e^{\frac 1 4\theta^2(\partial_t^2-\Delta^2)} u)$$ and
 $$  e^{\theta^2(\partial_t^2-\Delta^2)}(\Delta  u)= \frac 1 \theta (\theta\Delta e^{\frac 3 4\theta^2(\partial_t^2-\Delta^2)}) (e^{\frac 1 4\theta^2(\partial_t^2-\Delta^2)} u),$$
 we find that
  $$ \|  \partial_tu\|_{\mathcal{B}^{-\delta-2}_{[2],\infty,\infty}}+ \|  \Delta u\|_{\mathcal{B}^{-\delta-2}_{[2],\infty,\infty}} \leq C \|  u\|_{\mathcal{B}^{-\delta}_{[2],\infty,\infty}}.$$ 
  As $u=0$ at infinity, we   write
   $$ u=\int_0^{\infty} (-\partial_\eta) (e^{\eta^2(\partial_t^2-\Delta^2)}u)\, d\eta=2\int_0^{+\infty} \eta (\Delta^2-\partial_t^2)e^{\eta^2(\partial_t^2-\Delta^2)}u\, d\eta. $$
 We get
  \begin{equation*}\begin{split} e^{\theta^2(\partial_t^2-\Delta^2)}u=& 2\int_0^{+\infty} (\eta \Delta e^{\frac{\eta^2}2(\partial_t^2-\Delta^2)}) (e^{(\theta^2+\frac{\eta^2}2)(\partial_t^2-\Delta^2)}\Delta u)\, d\eta\\&-2\int_0^{+\infty}  (\eta \partial_t e^{\frac{\eta^2}2(\partial_t^2-\Delta^2)}) (e^{(\theta^2+\frac{\eta^2}2)(\partial_t^2-\Delta^2)}\partial_t u)\, d\eta,
 \end{split}\end{equation*} 
  so that
    \begin{equation*}\begin{split} \|e^{\theta^2(\partial_t^2-\Delta^2)}u\|_\infty\leq & C\int_0^{+\infty} \|e^{(\theta^2+\frac{\eta^2}2)(\partial_t^2-\Delta^2)}\Delta u\|_\infty +\|e^{(\theta^2+\frac{\eta^2}2)(\partial_t^2-\Delta^2)}\partial_t u\|_\infty\, d\eta
    \\ \leq& C (\|  \partial_tu\|_{\mathcal{B}^{-\delta-2}_{[2],\infty,\infty}}+ \|  \Delta u\|_{\mathcal{B}^{-\delta-2}_{[2],\infty,\infty}} ) \int_0^{+\infty} (\theta^2+\frac{\eta^2}2)^{-\frac{2+\delta} 4}\, d\eta
    \\=& C \theta^{-\delta/2} (\|  \partial_tu\|_{\mathcal{B}^{-\delta-2}_{[2],\infty,\infty}}+ \|  \Delta u\|_{\mathcal{B}^{-\delta-2}_{[2],\infty,\infty}} ) \int_0^{+\infty} (1+\frac{\eta^2}2)^{-\frac{2+\delta} 4}\, d\eta
 \end{split}\end{equation*} 
   and finally
   \begin{equation*}\|  u\|_{\mathcal{B}^{-\delta}_{[2],\infty,\infty}}\leq C 
 ( \|  \partial_tu\|_{\mathcal{B}^{-\delta-2}_{[2],\infty,\infty}}+ \|  \Delta u\|_{\mathcal{B}^{-\delta-2}_{[2],\infty,\infty}}). \tag*{\qedhere}\end{equation*}   \end{proof}

   \subsection*{Proof of Theorem \ref{theo6}.}
\begin{proof} We write 
   $$ f=\int_0^{\infty} (-\partial_\eta) (e^{\eta^2(\partial_t^2-\Delta^2)}f)\, d\eta=-2 (\partial_t-\Delta)\int_0^{+\infty} \eta (\Delta+\partial_t)e^{\eta^2(\partial_t^2-\Delta^2)}f\, d\eta. $$
   Let $$u= -2  \int_0^{+\infty} \eta (\Delta+\partial_t)e^{\eta^2(\partial_t^2-\Delta^2)}f\, d\eta. $$
We have
  \begin{equation*}\begin{split} \|e^{\theta^2(\partial_t^2-\Delta^2)}u\|_\infty\leq & 2\int_0^{+\infty} \|(\eta \Delta e^{\frac{\eta^2}2(\partial_t^2-\Delta^2)}) (e^{(\theta^2+\frac{\eta^2}2)(\partial_t^2-\Delta^2)} f)\|_\infty\, d\eta\\&+2\int_0^{+\infty}  \|(\eta \partial_t e^{\frac{\eta^2}2(\partial_t^2-\Delta^2)}) (e^{(\theta^2+\frac{\eta^2}2)(\partial_t^2-\Delta^2)}f)\|_\infty\, d\eta,
\\ \leq & C\int_0^{+\infty} \|e^{(\theta^2+\frac{\eta^2}2)(\partial_t^2-\Delta^2)} f\|_\infty\, d\eta
    \\ \leq& C \|  f\|_{\mathcal{B}^{-\delta-2}_{[2],\infty,\infty}} \int_0^{+\infty} (\theta^2+\frac{\eta^2}2)^{-\frac{2+\delta} 4}\, d\eta
    \\=& C \theta^{-\delta/2} \|  f\|_{\mathcal{B}^{-\delta-2}_{[2],\infty,\infty}} \int_0^{+\infty} (1+\frac{\eta^2}2)^{-\frac{2+\delta} 4}\, d\eta
 \end{split}\end{equation*} 
 and finally
   \begin{equation*}\|  u\|_{\mathcal{B}^{-\delta}_{[2],\infty,\infty}}\leq C 
 \|  f\|_{\mathcal{B}^{-\delta-2}_{[2],\infty,\infty}}. \tag*{\qedhere}\end{equation*}   \end{proof}
 

\end{document}